\documentclass[12pt]{article}%
\usepackage{amsmath}
\usepackage{amsfonts}
\usepackage{amssymb}
\usepackage{graphicx}%
\providecommand{\U}[1]{\protect\rule{.1in}{.1in}}
\providecommand{\U}[1]{\protect\rule{.1in}{.1in}}
\newtheorem{theorem}{Theorem}

\newtheorem{definition}[theorem]{Definition}

\newtheorem{lemma}[theorem]{Lemma}

\newtheorem{remark}[theorem]{Remark}

\newenvironment{proof}[1][Proof]{\noindent\textbf{#1.} }{\ \rule{0.5em}{0.5em}}

\begin{document}

\title{Regularity at infinity of Hadamard manifolds with respect to some elliptic
operators and applications to asymptotic Dirichlet problems}
\author{Jaime Ripoll -- Miriam Telichevesky}
\date{November, 2012}
\maketitle

\begin{abstract}
Let $M$ be Hadamard manifold with sectional curvature $K_{M}\leq-k^{2}$,
$k>0$. Denote by $\partial_{\infty}M$ the asymptotic boundary of $M$. We say
that $M$ satisfies the strict convexity condition (SC condition) if, given
$x\in\partial_{\infty}M$ and a relatively open subset $W\subset\partial
_{\infty}M$ containing $x$, there exists a $C^{2}$ open subset $\Omega\subset
M$ such that $x\in\operatorname*{Int}\left(  \partial_{\infty}\Omega\right)
\subset W$ and $M\setminus\Omega$ is convex. We prove that the SC condition
implies that $M$ is regular at infinity relative to
the operator
\[
\mathcal{Q}\left[  u\right]  :=\mathrm{{\,div\,}}\left(  \frac{a(|\nabla
u|)}{|\nabla u|}\nabla u\right)  ,
\]
subject to some conditions. It follows that under
the SC condition, the Dirichlet problem for the minimal hypersurface and the
$p$-Laplacian ($p>1$) equations are solvable for any prescribed continuous
asymptotic boundary data. It is also proved that if $M$ is rotationally
symmetric or if $\inf_{B_{R+1}}K_{M}\geq-e^{2kR}/R^{2+2\epsilon},\,\,R\geq
R^{\ast},$ for some $R^{\ast}$ and $\epsilon>0,$ where $B_{R+1}$ is the
geodesic ball with radius $R+1$ centered at a fixed point of $M,$ then $M$
satisfies the SC condition.
\end{abstract}

\section{Introduction}

\qquad Let $M$ be Hadamard manifold (complete, simply connected Riemannian
manifold) with sectional curvature $K_{M}\leq-k^{2},$ $k>0$. Denote by
$\overline{M}$ the compactification of $M$ in the cone topology and by
$\partial_{\infty}M$ the asymptotic boundary of $M$ (see for instance
\cite{EO}). This upper bound hypothesis on the sectional curvature of $M$
will be assumed throughout the paper.

%\bigskip
We consider the elliptic diferential operator $\mathcal{Q}$ defined in the
Sobolev space $W_{\mathrm{{loc}}}^{1,p}(M)$, for some $p\geq1$, given by
\begin{equation}
\mathcal{Q}\left[  u\right]  :=\mathrm{{\,div\,}}\left(  \frac{a(|\nabla
u|)}{|\nabla u|}\nabla u\right)  =0 \label{Qdefinition}%
\end{equation}
where $a\in C^{1}\left(  \left[  0,\infty\right)  \right)  $ satisfies
\begin{align}
&  a(0)=0,a^{\prime}(s)>0\text{ for all }s>0;\label{a1}\\
&  a(s)\leq C\left(  s^{p-1}+1\right)  \text{ for all }s\in\lbrack
0,+\infty),\text{ for some constant }C>0;\label{a2}\\
&  \text{there exist }q>0\text{ and }\delta>0\text{ such that }a(s)\geq
s^{q},\text{ }s\in\left[  0,\delta\right]  \label{a3}%
\end{align}
In the above PDE equation $\operatorname{div}$ and $\nabla$ denote the
divergence and the gradient in $M.$

The asymptotic Dirichlet problem with boundary condition $\varphi\in
C^{0}\left(  \partial_{\infty}M\right) $ consists in finding a solution $u\in
W_{\mathrm{{loc}}}^{1,p}(M)\cap C^{0}\left(  M\right)  $ of $\mathcal{Q}=0$ in
$M$ that extends continuously to $\partial_{\infty}M$ and $u|_{\partial
_{\infty}M}=\varphi$. By a solution $u$ of $\mathcal{Q}=0$ in $M$ we mean a
weak solution, that is, $u$ is in the Sobolev space $W_{\operatorname*{loc}%
}^{1,p}(M)$ and satisfies
\[
\int_{M}\left\langle \frac{a\left(  \left\vert \nabla u\right\vert \right)
}{|\nabla u|} \nabla u ,\nabla\xi\right\rangle =0
\]
for all all $\xi\in W_{0}^{1,p}(M)$.

In the Laplacian case this problem has been intensively investigated on the
last three decades (see, for instance, \cite{A}, \cite{Choi}, \cite{Hs},
\cite{Ne}, \cite{SY} and references therein).

\bigskip
The notion of \textit{regularity }of the boundary of a domain $\Omega\subset
M$ is fundamental in elliptic PDE theory to prove that a solution of the
Dirichlet problem of an elliptic PDE assumes a given value at the boundary of
$\Omega$ (see \cite{GT}). To deal with the asymptotic Dirichlet problem in $M$
we extend this notion to the asymptotic boundary $\partial_{\infty}M$ of $M$
as follows.

%\bigskip
We say that $M$ is regular at infinity with respect to $\mathcal{Q}$ if, given
$C>0,$ $x\in\partial_{\infty}M,$ and a relatively open subset $W\subset
\partial_{\infty}M$ containing $x,$ there are an open subset $\Omega\subset M$
such that $x\in\operatorname*{Int}\left(  \partial_{\infty}\Omega\right)
\subset W,\ $sub and supersolutions $\sigma,\Sigma\in C^{0}\left(  M\right)  $
of $\mathcal{Q}=0$ in $M$ (called \textit{barriers }at $x)$ such that
$\sigma\leq0\leq\Sigma$, $\lim_{p\rightarrow x}\sigma(p)=\lim_{p\rightarrow
x}\Sigma(p)=0$ and $\sigma|_{M\setminus\Omega}\leq-C$ and $\Sigma
|_{M\setminus\Omega}\geq C$ (see Section \ref{preliminaries} and Definition
\ref{reg} for more details).

%\bigskip
We prove that if $M$ is regular at infinity with respect to $\mathcal{Q}$ and if:

\begin{enumerate}
\item[{\bf(a)}] there is an exhaustion of $M$ by an increasing sequence of
$C^{\infty}$ bounded subdomains $\Omega_{k}\subset M$ such that the Dirichlet
problem for $\mathcal{Q}=0$ is solvable in $\Omega_{k}$ for any boundary data
which is in $C^{\infty}\left( \partial\Omega_{k}\right) $,

\item[{\bf(b)}] sequences of solutions with uniformly bounded $C^{0}$ norm are
compact in relatively compacts subsets of $M$,
\end{enumerate}

\noindent then the asymptotic Dirichlet problem for $\mathcal{Q}$ is solvable
for any continuous boundary data (Theorem \ref{abcd}).

%\bigskip
In view of the classical and more recent results on geometric quasilinear
elliptic PDE theory, conditions \textbf{(a)} and \textbf{(b)} hold in a large
class of elliptic PDE's, reducing then the solvability of this problem to the
regularity of $M$ at infinity with respect to the operator $\mathcal{Q}$.

\bigskip
From the well known work of H. Choi \cite{Choi}, it follows that if $M$
satisfies the convex conic neighborhood condition, namely, any two distinct
points of $\partial_{\infty}M$ can be separated by $C^{2}$ convex open
disjoint neighborhoods of these points in $\overline{M},$ then any point at
infinity is regular for the Laplace operator (the $2-$Laplacian$)$; in
particular, the asymptotic Dirichlet problem for the Laplace equation $\Delta
u=0$ is solvable for any given continuous boundary data $\varphi$ at
$\partial_{\infty}M.$ Since Choi's proof using the convex conic neighborhood
condition is heavily based on the linearity of the Laplacian operator, it does
not apply directly to quasi linear elliptic PDE's, as the $p-$Laplacian or the
minimal hypersurface PDE.

%\bigskip
In the present work we propose a different notion of convexity; informally
speaking, when one can extract from $\overline{M}$ a neighborhood of any point
of $\partial_{\infty}M$ such that what remains is still convex (as it happens
with strictly convex bounded domains). The precise definition is:

\begin{definition}
Let $M$ be a Hadamard manifold. We say that $M$ satisfies the \emph{strict
convexity condition (SC condition)} if, given $x\in\partial_{\infty}M$ and a
relatively open subset $W\subset\partial_{\infty}M$ containing $x,$ there
exists a $C^{2}$ open subset $\Omega\subset\overline{M}$ such that
$x\in\operatorname*{Int}\left(  \partial_{\infty}\Omega\right)  \subset W,$
where $\operatorname*{Int}\left(  \partial_{\infty}\Omega\right)  $ denotes
the interior of $\partial_{\infty}\Omega$ in $\partial_{\infty}M,$ and
$M\setminus\Omega$ is convex.
\end{definition}

As we shall see, the SC condition gives an a priori control of the behavior at
infinity of the solutions of divergence form quasi-linear elliptic PDEs
\eqref{Qdefinition} subject to the conditions \eqref{a1}--\eqref{a3}. In fact
we prove that if $M$ satisfies the SC condition then it is regular at infinity
with respect to $\mathcal{Q}$ (Theorem \ref{barriers}). As a consequence,
under the SC condition, when $\mathcal{Q}\ $satisfies conditions \textbf{(a)}
and \textbf{(b)} then the asymptotic Dirichlet problem for $\mathcal{Q}$ is
solvable for any continuous boundary data (Theorem \ref{scab}).

%\bigskip
We observe that the minimal hypersurface PDE $\mathcal{M}=0$ and the
$p-$Laplacian PDE $\Delta_{p}=0,$ $p>1,$ are special cases of
$\mathcal{Q}=0$, where $a(t)=t/\sqrt{1+t^{2}}$ and $a(t)=t^{p-1}$,
respectively. It is easy to see that they both satisfy conditions (\ref{a1}),
(\ref{a2}), (\ref{a3}) and, as we shall see later, both $\mathcal{M}$ and
$\Delta_{p}$ satisfy conditions \textbf{(a)} and \textbf{(b)}. Then, from
Theorem \ref{barriers}, we obtain that if $M$ satisfies the SC condition then
the Dirichlet problem with prescribed continuous asymptotic data is solvable
for $\mathcal{M}$ and $\Delta_{p},$ $p>1$ (Theorem \ref{main})$.$

\bigskip
To show the extent of the above results, we first prove that if $M$ has a
rotationally symmetric metric then it satisfies the SC condition (Theorem
\ref{rotsim}). By an adaptation of a nice construction due to Borb\'{e}ly
\cite{B1} we also prove that $\inf_{B_{R+1}}K_{M}\geq-e^{2kR}/R^{2+2\epsilon
},\,\,R\geq R^{\ast},$ for some $R^{\ast}$ and $\epsilon>0,$ where $B_{R+1}$
is the geodesic ball with radius $R+1$ centered at a fixed point of $M,$ our
SC condition is satisfied (Theorem \ref{bor}). Remark that $2$-dimensional
Hadamard manifolds (under the assumption $K_{M}\leq-k^{2}$) always satisfy the
SC condition, since any two points of $\partial_{\infty}M$ can be connected by
a geodesic.

The classes of Hadamard manifolds $M$ that we have considered essentially
satisfy both Choi's convex conic condition and our SC condition. It may then
be possible that these conditions are actually equivalent. So far, this
remains an open problem. We also observe that if no convexity is required then
it was constructed by Ancona \cite{Anc} and Borb\'ely \cite{B2} examples
of $3-$ dimensional Hadamard manifolds $M$ with $K_{M}\leq-1$ for which the
asymptotic Dirichlet problem for the Laplace equation is not solvable for all
continuous non constant boundary data. In \cite{H2} I. Holopainen extended Borb\'ely's resulto to the $p-$Laplacian PDE and, quite recently, the first author of the present article and I. Holopainen \cite{HR} extended these nonsolvability results to the class of PDEs \eqref{Qdefinition} including, particularly, the minimal hypersurface PDE. We finally mention that the hypothesis $K_{M}\leq-k^{2}<0$ cannot be relaxed to $K_{M}<0,$ even in dimension $2.$ This follows from the results of M. Rigoli and A. Setti in \cite{RS}\ where the authors study the PDE
\eqref{Qdefinition} focusing in nonexistence Liouville's type theorems (see
Remark \ref{KM} below).

\section{\label{preliminaries}The asymptotic Dirichlet problem for Hadamard
manifolds which satisfy the strict convexity condition}

Let $M$ be a Hadamard manifold. We say that a function $\Sigma\in C^{0}\left(
M\right)  $ is a supersolution for $\mathcal{Q}$ if, given a bounded domain
$U\subset M,$ if $u\in C^{0}\left(  \overline{U}\right)  $ is a solution of
$\mathcal{Q}=0$ in $U$, the condition $u|_{\partial U}\leq\Sigma|_{\partial
U}$ implies that $u\leq\Sigma|_{U}$.

Given $x\in\partial_{\infty} M$ and an open subset $\Omega\subset M$ such that
$x\in\partial_{\infty} \Omega$, an \emph{upper barrier for $\mathcal{Q}$
relative to $x$ and $\Omega$ with height $C$} is a function $\Sigma\in
C^{0}(M)$ such that

\begin{enumerate}
\item[(i)] $\Sigma$ is a supersolution for $\mathcal{Q}$;

\item[(ii)] $\Sigma\ge0$ and $\displaystyle\lim_{p\in M,\,p\to x} \Sigma(p) =0$, the limit with respect to the cone topology;
\item [(iii)] $\Sigma_{M\setminus\Omega} \ge C$.
\end{enumerate}

Similarly, we define subsolutions and lower barriers.

\begin{definition}
\label{reg} Let $M$ be a Hadamard manifold. We say that $M$ is \emph{regular
at infinity with respect to the differential operator $\mathcal{Q}$} if, given
$C>0$, $x\in\partial_{\infty}M$ and an open subset $W\subset\partial_{\infty
}M$ with $x\in W$, there exist an open set $\Omega\subset M$ such that
$x\in\operatorname*{Int} \partial_{\infty} \Omega\subset W$ and upper and
lower barriers $\Sigma,\sigma:M\rightarrow\mathbb{R}$ relatives to $x$ and
$\Omega$, with height $C$.
\end{definition}

Compare this definition with Definition 2.6 and Theorem 2.7 in \cite{Choi} for
the case of the Laplace operator and also with Theorem 3.3 and Definition 3.4
in \cite{HV} for the case of the $p-$Laplacian.

The next lemma give candidates to supersolutions of $\mathcal{Q}=0$, and hence
to upper barriers. Although it is a well known result, we write the proof for
reader's convenience.

\begin{lemma}
\label{QvleQu} Let $U\subset M$ be an open set and $v\in C^{0}(U)\cap
W^{1,p}_{loc}(U)$ such that, for every $\xi\in C^{\infty}_{0}(U)$ with $\xi
\ge0$, it holds that
\begin{equation}
\label{weaksupersolution}\int_{U} \left\langle \frac{a(|\nabla v|}{|\nabla v|}
\nabla v, \nabla\xi\right\rangle dx \ge0
\end{equation}
($v$ is frequently known as a \emph{weak supersolution of $\mathcal{Q}=0$}).
Then $v$ is a supersolution of $\mathcal{Q}=0$. In particular, if $v\in
C^{2}(U)$ satisfies $\mathcal{Q}[v]\le0$, then $v$ is a supersolution.
\end{lemma}

\begin{proof}
We first notice that inequality \eqref{weaksupersolution} holds for all
$0\le\xi\in C^{\infty}_{0}(U)$ if and only if it holds for all $\xi\in
W^{1,p}_{0}(U)$, with $\xi\ge0$ a.e..

In order to prove that $v$ is a supersolution, let $B\subset\subset U$ and
$u\in C^{0}(\overline{B}) \cap W^{1,p}(B)$ such that $u\le v$ on $\partial B$.
We must prove that $u\le v$ in $B$. For, notice that inequality
\eqref{weaksupersolution} holds for the particular choice of $\xi$ given by
$\xi=(u-v)_{+}:=\max\{u-v,0\}$, by the remark above. On the other hand, since
$u$ is a solution of $\mathcal{Q}=0$, it holds that
\[
\int_{B} \left\langle \frac{a(|\nabla u|}{|\nabla u|} \nabla u, \nabla
\xi\right\rangle dx =0.
\]
Therefore we may combine both expressions and obtain

\begin{equation}
\label{integralnegativa}\int_{B} \left\langle \frac{a(|\nabla u|)}{|\nabla
u|}\nabla u - \frac{a(|\nabla v|)}{|\nabla v|}\nabla v, \nabla u - \nabla v
\right\rangle dx \le0.
\end{equation}

On the other hand,

\begin{align*}
\left\langle \frac{a(|\nabla u|)}{|\nabla u|}\nabla u - \frac{a(|\nabla
v|)}{|\nabla v|}\nabla v, \nabla u - \nabla v \right\rangle  = a(|\nabla u|)
|\nabla u|^{2} \\- \frac{a(|\nabla u|)}{|\nabla u|} \langle\nabla u, \nabla
v\rangle - \frac{a(|\nabla v|)}{|\nabla v|} \langle\nabla u, \nabla v \rangle  +
a(|\nabla v|)|\nabla v|\ge \\  a(|\nabla u|)|\nabla u|-a(|\nabla u|)|\nabla v|
 - a(|\nabla v|)|\nabla u| +a(|\nabla v|)|\nabla v|=\\
 \left(a(|\nabla u|) - a(|\nabla v|) \right) \left( |\nabla u|-|\nabla
v|\right) ,
\end{align*}
where the inequality is implied by Cauchy-Schwarz inequality; and since $a$ is
increasing the last product must be $\ge0$. Hence the integrated function on
\eqref{integralnegativa} is nonnegative and since its integral is nonpositive,
it vanish a.e. on $B$. Since $a$ is injective, it follows that $\nabla u = \nabla v$ a.e. in $B$ which implies
that $(u-v)_{+}$ is constant a.e.. Since it is continuous, we conclude
that it actually is constant in $B$ and since it vanishes on the boundary,
$(u-v)_{+}=0$, which finishes the first part of the proof.

The second part is a consequence of integration by parts.
\end{proof}

\begin{theorem}
\label{abcd}Let $M$ be a Hadamard manifold with sectional curvature $K_{M}\leq-k^{2}<0$ with is regular at infinity with respect to $\mathcal{Q}$.
Assume moreover that

\begin{enumerate}
\item[{\bf(a)}] there is a sequence of bounded $C^{\infty}$ domains $\Omega
_{k}\subset\Omega,$ $k\in\mathbb{N},$ satisfying $\Omega_{k}\subset
\Omega_{k+1},$ $\cup\Omega_{k}=M,$ such that the Dirichlet problem for
$\mathcal{Q}=0$ is solvable in $\Omega_{k}$ for any $C^{\infty}$ boundary data,

\item[{\bf(b)}] sequences of solutions with uniformly bounded $C^{0}$ norm are
compact in relatively compacts subsets of $\Omega$.
\end{enumerate}

\noindent Then the asymptotic Dirichlet problem of $\mathcal{Q}$ is solvable
for any continuous boundary data.
\end{theorem}

\begin{proof}
Let $\phi\in C^{\infty}(M)\cap C^{0}(\overline{M})$ be a smooth extension of
$\varphi$. Condition \textbf{(a)} allows us to solve the Dirichlet problem
\[
\left\{
\begin{array}
[c]{l}\mathcal{Q}[u]=0\text{ in }\Omega_{k}, u\in W^{1,p}_{loc}(\Omega_{k})\cap
C^{0}(\overline{\Omega_{k}})\\
u|_{\partial\Omega_{k}}=\phi,
\end{array}
\right.
\]
finding a solution $u_{k}\in C^{0}(\overline{\Omega_{k}})$.

Condition \textbf{(b)} together with the diagonal method show that there
exists a subsequence of $(u_{k})$ (which we suppose to be $(u_{k})$)
converging uniformly on compact subsets of $M$ to a global solution of $\mathcal{Q}
=0$, which we denote by $u$. One needs to show that $u$ extends continuously
to $\partial_{\infty} M$ and satisfies $u|_{\partial_{\infty} M}=\varphi$.

For, let $x\in\partial_{\infty} M$ and $\varepsilon>0$. Since $\varphi$ is continuous, there exists an open neighborhood
$W\subset\partial_{\infty} M$ of $x$ such that $\varphi(y) < \varphi(x)
+\varepsilon/2$ for all $y\in W$. Furthermore, the regularity of $M$ at infinity with respect to $\mathcal{Q}$ implies that there exists an open subset $\Omega\subset M$ such that
$x\in\operatorname*{Int}\left( \partial_{\infty} \Omega\right)  \subset W$ and
$\Sigma: M \to\mathbb{R}$ upper barrier with respect to $x$ and $\Omega$ with
height $C:=\max_{\overline{M}} |\phi|$.

Defining
\[
v(q):=\Sigma(q)+\varphi(x)+\varepsilon,
\]
we claim that $u\le v$ in $\Omega$.

Since $\phi$ is continuous, $\exists k_{0} >>0$ such that $\phi(q) <
\varphi(x) + \varepsilon/2$ for all $q\in\partial\Omega_{k} \cap\Omega$, $k\ge
k_{0}$; we may chose $k_{0}$ such that $\Omega_{k_{0}}\cap\Omega\neq\emptyset$.

Let $V_{k}:=\Omega\cap\Omega_{k}$, $k\ge k_{0}$. Claim: $u_{k}\le v$ in
$V_{k}$. In fact, the inequality holds on $\partial V_{k}=\overline{\left(
\partial\Omega_{k}\cap\Omega\right) }\cup\overline{\left( \partial\Omega
\cap\Omega_{k}\right) }$: on $\partial\Omega_{k} \cap\Omega$, it is true due
to the choice of $k_{0}$; on $\partial\Omega\cap\Omega_{k}$, it holds because
$\Sigma\ge\max|\varphi|$ on $\partial\Omega$, which implies that $\Sigma\ge
u_{k}$, by the Comparison Principle (Lemma \ref{QvleQu}).

Also the Comparison Principle implies that $u_{k} \le v$ in $V_{k}$; since it
holds for all $k\ge k_{0}$, we have $u\le v$ on $\Omega$.

It is also possible to define $v_{-}:M \to\mathbb{R}$ by $v_{-}(q):=\varphi
(x)-\varepsilon- \Sigma(q)$ in order to obtain $u\ge v_{-}$ in $\Omega$. It
then holds that
\[
|u(q) - \varphi(x) | < \varepsilon+ \Sigma(q),\,\forall\, q\in\Omega,
\]
and hence
\[
\limsup_{p\to x} |u(p) - \varphi(x)| \le\varepsilon.
\]
The proof is complete, since $\varepsilon>0$ is arbitrary.
\end{proof}

\begin{remark}
We remark that in the bounded case, the regularity of the domain seems to
depend heavier on the operator. For instance, any bounded $C^{2}$ domain is
regular for the Laplace equation as, for the mean curvature operator, the
$C^{2}$ regularity is not enough, the domain has to be convex too.
\end{remark}

\begin{theorem}
\label{barriers} Let $M$ be a Hadamard manifold with sectional curvature
$K_{M}\leq-k^{2}$ satisfying the SC condition. Then $M$ is regular at infinity
with respect to $\mathcal{Q}$.
\end{theorem}

\begin{proof}
Let $C>0$ and $x\in W\subset\partial_{\infty}M$ be given. Since $\mathcal{Q}
\left[  -u\right]  =-\mathcal{Q}\left[  u\right]  $ it is enough to prove the
existence of barrier from above at $x.$ Since $M$ satisfies the SC condition,
there exists a $C^{2}$ open subset $\Omega$ of $M$ such that $x\in
\operatorname*{Int}\left(  \partial_{\infty}\Omega\right)  \subset W$ and such
that $M\setminus\Omega$ is convex. Let $s:\Omega\rightarrow\mathbb{R}$ be the
distance function to $\partial\Omega$. Since $M\setminus\Omega$ is convex and
$K_{M}\leq-k^{2}$, we may apply comparison theorems (see Theorems 4.2 and 4.3
of \cite{Choi}) in order to obtain the estimative
\begin{equation}
\Delta s\geq(n-1)k\tanh ks. \label{Deltas}
\end{equation}
On the other hand, since $a^{\prime}>0$, $a$ has an inverse function
$a^{-1}\in C^{1}\left(  \left[  0,\alpha\right)  \right)  $ where $\alpha=\sup
a\leq\infty$. Set $c=a(2C).$

We may then define a function $g:\left[  0,\infty\right)  \rightarrow
\mathbb{R},$ $g\in C^{2}\left(  \left(  0,\infty\right)  \right)  ,$ possibly
with $g(0)=\infty,$ by
\begin{equation}
g(s):=\int_{s}^{\infty}a^{-1}\left(  c\cosh^{1-n}kt\right)  dt. \label{g}
\end{equation}

We observe that, from the assumptions on $a$, the function $g$ is well
defined. In fact, let $\tau$ satisfy
\[
c\cosh^{1-n}k \tau=\delta.
\]
Assuming, without loss of generality, that $\delta<c$, such $\tau$ exists
because $c\cosh^{1-n}0=c$ and $\lim_{t\rightarrow+\infty}c\cosh^{1-n}kt=0$. We
have $a^{-1}(t)\leq t^{1/q}$ if $a(t)\in\lbrack0,\delta]$, and therefore
\begin{align*}
g(s)  &  =\int_{s}^{\tau}a^{-1}(c\cosh^{1-n}kt)dt+\int_{\tau}^{+\infty}
a^{-1}(c\cosh^{1-n}kt)dt\\
&  \leq a^{-1}(cs)(\tau-s)+\int_{x}^{+\infty}(c\cosh^{1-n}kt)^{\frac{1}{q}}dt\\
&  \leq a^{-1}(cs)\tau+(2c)^{\frac{1}{q}}\int_{\tau}^{+\infty}e^{-\frac{kt}
{q}}dt=a^{-1}(cs)\tau+\frac{(2c)^{\frac{k}{q}}}{q}e^{-k\tau}<+\infty
\end{align*}
for all $s>0$. Furthermore,
\[
g(0)>\int_{0}^{1}a^{-1}\left(  c\cosh^{1-n}kt\right)  dt\geq a^{-1}(c)=2C
\]
and $\lim_{s\rightarrow\infty}g(s)=0$. Therefore we may define $v:\Omega
\rightarrow\mathbb{R}$ as
\[
v(p):=g(s(p)),
\]
and we want to show that $\mathcal{Q}(v)\leq0$. Notice that
\[
\nabla v(p)=g^{\prime}(s(p))\nabla s(p)=-a^{-1}\left( c\cosh^{1-n}
ks(p)\right)  \nabla s
\]
and then $|\nabla v|=|g^{\prime}(s(p))|=a^{-1}\left( c\cosh^{1-n}ks(p)\right)
$; furthermore, it holds that $\nabla v/|\nabla v|=-\nabla s$. Combining the previous
expressions, we obtain

\begin{align*}
\mathcal{Q}(v)= &  \mathrm{{\,div\,}}\left(  a\left(  |g^{\prime}(s(p))|\nabla
s(p)\right)  \right) \\
= &  \mathrm{{\,div\,}}\left(  a\left(  -a^{-1}\left(  c\cosh^{1-n}
ks(p)\right)  \right)  \nabla s(p)\right) \\
= &  \mathrm{{\,div\,}}\left(  -c\cosh^{1-n}ks(p)\nabla s(p)\right) \\
= &  -(1-n)ck\cosh^{-n}ks(p)\sinh ks(p)\langle\nabla s(p),\nabla s(p)\rangle\\
& -c\cosh^{1-n}ks(p)\Delta s(p)\\
\leq &  (n-1)ck\cosh^{-n}ks(p)\sinh ks(p)\\
& -(n-1)c\cosh^{1-n}ks(p)k\tanh ks(p)=0,
\end{align*}
and hence, by Lemma \ref{QvleQu}, $v$ is a supersolution on $\Omega$.

To finish with the proof, define the global supersolution $\Sigma\in
C^{0}\left(  \overline{M}\right)  $ by
\[
\Sigma(p)=\left\{
\begin{array}
[c]{ll}
\min\left\{  v(p),C\right\}  & \text{if }p\in\Omega\\
C & \text{if }p\in\overline{M}\setminus\Omega,
\end{array}
\right.
\]
which is of course an upper barrier relative to $x$ and $\Omega$ with height
$C$.
\end{proof}

As a consequence of Theorems \ref{abcd} and \ref{barriers} we obtain

\begin{theorem}
\label{scab} Let $M$ be a Hadamard manifold with sectional curvature
satisfying $K_{M}\leq-k^{2}$. Assume that $M$ satisfies the SC condition and
that $\mathcal{Q}$ satisfies conditions \emph{\textbf{(a)}} and
\emph{\textbf{(b)}}. Then the asymptotic Dirichlet problem of $\mathcal{Q}$ is
solvable for any continuous boundary data.
\end{theorem}

\begin{remark}
\label{KM}Since in the two-dimensional case the hypothesis $K_{M}\leq k^{2}<0$
implies the SC condition, it is interesting to know if Theorem \ref{scab}
remains true if we just require $K_{M}<0$. We may see that if
\begin{equation}
0>\inf_{\partial B_{r}}K_{M}\geq-\frac{1}{r^{2}\log r}, \label{KMB}
\end{equation}
where $B_{r}$ is a geodesic ball centered at some fixed point of $M,$ for
$r\geq r_{0}$, $r_{0}$ large enough, then Bishof comparison theorem implies
that $\left\vert \partial B_{r}\right\vert \leq r\log r,$ where $\left\vert
\partial B_{r}\right\vert $ is the volume of $\partial B_{r}.$ It follows that
for $p\geq 2$ and $s$ large enough we have
\[
\int_{s}^{\infty}\frac{dr}{\left\vert \partial B_{r}\right\vert ^{\frac{1}{p-1}}}
\geq\int_{s}^{\infty}\frac{dr}{r\log r}=\infty.
\]
From Theorem A of \cite{RS} it follows that there is no entire bounded
solution of $\mathcal{Q}=0$ in $M$ besides the constant functions whenever the function $a$ that defines $\mathcal{Q}$ satisfies $a(t)\le Ct$ {\em near the origin} for some $C>0$. This is the case, for instance, for the $p-$laplacian, $p\ge 2$, and minimal PDEs. In particular, the asymptotic Dirichlet problem for those cases is
not solvable for any continuous boundary data in $\partial_{\infty}M$ if
$K_{M}$ satisfies (\ref{KMB}). Hence, Theorem \ref{scab} is false if one
replaces the hypothesis $K_{M}\leq-k^{2}<0$ by $K_{M}<0.$

The nonexistence condition (\ref{KMB}) for the curvature is sharp for the
Laplacian and the minimal surface PDEs in the sense that if
\begin{equation}
\sup_{\partial B_{r}}K_{M}\leq-\frac{1+\varepsilon}{r^{2}\log r}\text{ (}\dim
M=2) \label{KE}
\end{equation}
for some $\varepsilon>0$ then it was proved by R. W. Neel in \cite{Ne} that
the asymptotic Dirichlet problem for the Laplacian PDE is solvable for any
continuous boundary data. In \cite{RT} it is proved that the same hypothesis
(\ref{KE}) on the curvature of a rotationally symmetric $M$ also implies the
solvability of the minimal surface equation in $M$ for any prescribed
continuous boundary data. We note that these remarks in the case of the
minimal PDE give a partial answer to a question formulated in \cite{GR}.
\end{remark}

\begin{theorem}
\label{main} Let $M$ be a Hadamard manifold with sectional curvature
satisfying $K_{M}\leq-k^{2}$ and assume that $M$ satisfies the SC condition.
Then the asymptotic Dirichlet problems for the minimal hypersurface
$\mathcal{M}=0$ PDE and the $p-$Laplacian $\Delta_{p}=0$ PDE, $p>1$, are
uniquely solvable for any given continuous boundary data at infinity. The
solution is in the classical sense in the minimal case and in the weak sense
in the $p-$Laplacian case. Precisely: Given $\varphi\in C^{0}\left(
\partial_{\infty}M\right)  $ there are solutions $u\in C^{\infty}\left(
M\right)  \cap C^{0}\left(  \overline{M}\right)  $ of $\mathcal{M}\left[
u\right]  =0$ and $v\in C^{1,\alpha}(M)\cap C^{0}\left(  \overline{M}\right)
$ of $\Delta_{p}\left[  v\right]  =0$ such that $u|_{\partial_{\infty}
M}=v|_{\partial_{\infty}M}=\varphi.$
\end{theorem}

\begin{proof}
Due to Theorem \ref{barriers} above one only has to prove that conditions
\textbf{(a)} and \textbf{(b)} are satisfied for $\mathcal{M}$ and $\Delta_{p}
$. Fixing a point $o\in M$ we take $\Omega_{k}:=B_{k}$, the geodesic ball of
$M$ with radius $k$ centered at $o.$ Choose $k$ and let $\psi\in C^{\infty
}\left(  \partial B_{k}\right)  $ be given. In the $p-$Laplacian case, the
same arguments used for proving the existence of solutions of the
$p-$Laplacian PDE in $\mathbb{R}^{n}$ with prescribed $C^{\infty}$ boundary
data apply to prove the existence of a weak solution $u_{k}\in C^{1,\alpha
}\left(  \overline{B_{k}}\right)  $ of $\Delta_{p}=0$ in $B_{k}$ such that
$u_{k}|_{\partial B_{k}}=\psi$ for some $\alpha\in\left(  0,1\right)  $ which
depends only on $p.$ This proves that condition \textbf{(a)} is satisfied for
$\Delta_{p}$. The interior gradient estimates of \cite{K} and standard
elliptic PDE theory implies that uniformly $C^{0}$ bounded sequences of
solutions of $\Delta_{p}=0$ are compact on relatively compact subsets of $M$
proving that condition \textbf{(b)} is satisfied for $\Delta_{p}$. In the case
of the minimal PDE, condition \textbf{(a)} follows from Theorem 1 of
\cite{DHL} and elliptic regularity theory. Condition \textbf{(b)} follows from
Theorem 1.1 of \cite{S} and standard arguments from PDE elliptic theory. This
proves the theorem.
\end{proof}

\section{Applications}

\qquad In this section, we first show that rotationally symmetric manifolds
with sectional curvature $K_{M}\leq-k^{2}$ satisfy the strict convexity
condition. The proof of this fact is constructive and there is almost no
difference between the $3$-dimensional and the $n$-dimensional case, $n\geq3$,
as we shall see in the final of the constructions. The existence results
ensured by Theorem \ref{main} in rotationally manifolds, for the minimal and
$p-$Laplacian PDE's, $p>1,$ have already been obtained by different authors in
different papers, some of them with weaker conditions on the sectional
curvature (see \cite{GR}, \cite{EFR} and \cite{V2}).

Let $M$ be a $(n+1)$-dimensional rotationally symmetric Hadamard manifold with
center $o$. We may treat $M$ as the set $$\mathbb{R}^{n+1}=\{(x_{1}
,\dots,x_{n+1})|x_{i}\in\mathbb{R}\}$$ parametrized by $r$, the distance to
$(0,\dots,0)=o$, and $\theta_{1},\dots,\theta_{n}$, with
\begin{align}
\label{sphericalcoordinates} &  x_{1} = r\sin\theta_{1}\sin\theta_{2}\dots
\sin\theta_{n}\nonumber\\
&  x_{2} = r\cos\theta_{1}\sin\theta_{2}\dots\sin\theta_{n}\nonumber\\
&  x_{3} = r\cos\theta_{2}\sin\theta_{3}\dots\sin\theta_{n}\nonumber\\
&  \cdots\\
&  x_{n} = r\cos\theta_{n-1}\sin\theta_{n}\nonumber\\
&  x_{n+1} = r\cos\theta_{n},\nonumber
\end{align}
endowed with the metric
\begin{equation}
dr^{2} + f(r)^{2} d\omega^{2}, \label{rotsymmet}
\end{equation}
where $d\omega^{2}$ is the standard metric of $\mathbb{S}^{n}$. We use the
notation $M=(\mathbb{R}^{n+1},f)$ for $M$ if it is rotationally symmetric with
metric given by \eqref{rotsymmet}.

\begin{remark}
Notice that:

\begin{enumerate}
\item[(i)] The angles at the origin $o$ are the same whatever is the
function $f$ that gives the metric.

\item[(ii)] The parameter $\theta_{n}$ is the angle with the
direction $(0,\dots,0,1)$.
\end{enumerate}
\end{remark}

The vector fields defined by
\begin{equation}
U:=\frac{\partial}{\partial r}, \,\,V_{i}:=\frac{\partial}{\partial\theta_{i}}
\end{equation}
give a local frame for $M\setminus\{o\}$. If we use the notation
$g_{00}:=\langle U, U\rangle$ and $g_{ii}:= \langle V_{i},V_{i}\rangle$, it is
easy to see that
\begin{align}
\label{gijrotsym} &  g_{ii}=f(r)^{2}\sin^{2} \theta_{i+1}\sin^{2} \theta_{i+2}
\dots\sin^{2} \theta_{n}, 1\le i\le n-1,\nonumber\\
&  g_{00}=1, g_{nn}=f(r)^{2} \text{ and }g_{ij}=0\text{ if }i\neq j.
\end{align}

We now fix $\gamma:[0,+\infty)\to M$ a unit speed geodesic ray with
$\gamma(0)=o$. Since any geodesic ray with the same properties may be mapped
by an isometry on $\gamma$, we may assume without loss of generality that
$\gamma(s)=(0,\dots,0,s)$, $s\ge0$.

Let $r(t),\theta(t)$ be given by the solution of the system
\begin{equation}
\label{rtheta}\left\{
\begin{array}
[c]{lll}
r^{\prime}(t)=\cosh kR\sin\theta(t) &  & \\
\theta^{\prime}(t)=\displaystyle\frac{k\sinh kR}{\sinh^{2}(kr(t))} &  & \\
r(0)=R,\,\,\theta(0)=0, &  &
\end{array}
\right.
\end{equation}
and let $S_{R}\subset\mathbb{R}^{n+1}$ be the hypersurface parametrized by
\begin{align}
\label{SRparametrization}(t,\theta_{1},\dots,\theta_{n-1})\mapsto &  \left(
r(t)\sin\theta_{1}\sin\theta_{2}\dots\sin\theta_{n-1}\sin\theta(t),\right.
\nonumber\\
&  r(t)\cos\theta_{1}\sin\theta_{2}\dots\sin\theta_{n-1}\sin\theta
(t),\nonumber\\
&  r(t)\cos\theta_{2}\sin\theta_{3}\dots\sin\theta_{n-1}\sin\theta
(t),\nonumber\\
&  \left.  \dots,r(t)\cos\theta_{n-1}\sin\theta(t),r(t)\cos\theta(t)\right)  ,
\end{align}
where $r(t)$ and $\theta(t)$ are given by \eqref{rtheta}.

\begin{lemma}
\label{SRhyperbolic} If $M=\left(  \mathbb{R}^{n+1},\sinh(kr)/k\right)  $,
i.e., $M$ is the hyperbolic space with constant sectional curvature equal to
$-k^{2}$, then the hypersurface $S_{R}$ is the hyperbolic totally geodesic
hypersurface that intersects $\gamma$ ortogonally at $(0,\dots,0,R)$.
Moreover,
\begin{equation}
\label{abSR}\partial_{\infty}S_{R}=\left\{  \tilde{\gamma}(\infty
)|\tilde{\gamma}(0)=o,\sphericalangle_{o}(\tilde{\gamma}^{\prime}(0)
,\gamma^{\prime}(0))=\arccos\left(  \tanh kR\right)  \right\}  .
\end{equation}

\end{lemma}

The preceding Lemma (that will be proved after the next proposition) shows
that $S_{R}$ exists and divides $\mathbb{R}^{n+1}$ in two connected
components. These ones are both convex if $\mathbb{R}^{n+1}$ is endowed with
the hyperbolic metric, since the second fundamental form of $S_{R}$ vanishes.
The next proposition shows that $S_{R}$ has the desired properties if we
replace the hyperbolic metric by a rotationally symmetric one, if $K_{M}
\leq-k^{2}$.

We assume that $R>0$, and we denote by $\Omega^{\prime}$ the connected
component of $\mathbb{R}^{n+1}\setminus S_{R}$ that contains the point
$(0,0,\dots,0)$.

\begin{lemma}
\label{SRgeneral} Let $S_{R}\subset\mathbb{R}^{n+1}$ be the hypersurface
defined as above, with $R>0$. If $M=\left(  \mathbb{R}^{n+1},f\right)  $ has
sectional curvature $K_{M}$ satisfying $K_{M} \le-k^{2}$, then $\Omega
^{\prime}$ is convex.
\end{lemma}

\begin{proof}
Since $S_{R}$ is parametrized by \eqref{SRparametrization}, the vector fields
defined by
\[
T:=\frac{\partial}{\partial t}, V_{1}, V_{2}, \dots, V_{n-1}
\]
give an orthogonal frame of $T_{p}S_{R}$ for any $p\in S_{R}$. Furthermore,
notice that $T=r^{\prime}(t)U+\theta^{\prime}(t)V_{n}$.

It is a straightforward calculation to verify that
\begin{equation}
N(t,\theta_{1},\dots,\theta_{n-1}):=-f(r(t))\theta^{\prime}(t)U+\frac
{r^{\prime}(t)}{f(r(t))}V_{n} \label{normalSR}
\end{equation}
is normal to $S_{R}$ pointing to $\Omega^{\prime}$. In order to show that
$\Omega^{\prime}$ is convex, we use Theorem 4.1 of \cite{Choi}: It suffices to
prove that the second fundamental form of $S_{R}$ with respect to $N$ is
nonnegative. It is then sufficient to show that
\[
\langle\nabla_{T}T,N\rangle,\,\langle\nabla_{V_{i}}V_{i},N\rangle
\geq0,\,\forall i\in\{1,\dots,n-1\},
\]
where $\nabla$ is the Levi-Civita connection of $\left(  \mathbb{R}
^{n+1},f\right)  $.

We start with $\langle\nabla_{T} T, N\rangle$:
\begin{align}
\label{nablaTTN1}\langle\nabla_{T} T, N\rangle &  = -f\theta^{\prime}
\langle\nabla_{T} T, U\rangle+ \frac{r^{\prime}}{f}\langle\nabla_{T} T,
V\rangle\nonumber\\
&  = -f\theta^{\prime}T\langle T, U\rangle+f\theta^{\prime}\langle T,
\nabla_{T} U\rangle+ \frac{r^{\prime}}{f} T\langle T, V\rangle- \frac
{r^{\prime}}{f}\langle T, \nabla_{T} V\rangle\nonumber\\
&  = -f\theta^{\prime}r^{\prime\prime}+f\theta^{\prime}\langle T, \nabla_{T}
U\rangle+ \frac{r^{\prime}}{f}\left(  f^{2}\theta^{\prime}\right)  ^{\prime}-
\frac{r^{\prime}}{f}\langle T, \nabla_{T} V\rangle.
\end{align}

Notice that
\begin{align}
\label{nablaTTN2}\langle T, \nabla_{T} U\rangle &  = r^{\prime2} \langle U,
\nabla_{U} U\rangle+ r^{\prime}\theta^{\prime}\left(  \langle U, \nabla_{V}
U\rangle+ \langle V, \nabla_{U} U\rangle\right)  + \theta^{\prime2} \langle V,
\nabla_{V} U\rangle\nonumber\\
&  = 0 + 0 + 0 + \theta^{\prime2} \frac{1}{2}U\left(  \Vert V\Vert^{2}\right)
= \theta^{\prime2}ff_{r}.
\end{align}
and
\begin{align}
\label{nablaTTN3}\langle T, \nabla_{T} V \rangle &  = r^{\prime2} \langle U,
\nabla_{U} V\rangle+ r^{\prime}\theta^{\prime}\left(  \langle U, \nabla_{V}
V\rangle+ \langle V, \nabla_{U} V\rangle\right)  + \theta^{\prime2} \langle V,
\nabla_{V} V\rangle\nonumber\\
&  = r^{\prime}\theta^{\prime}\left(  -\frac{1}{2}U\left(  \Vert
V\Vert^{2}\right)  +\frac{1}{2} U \left(  \Vert V\Vert^{2}\right)  \right)  +
\theta^{\prime2}\frac{1}{2} V\left(  \Vert V \Vert^{2}\right)  = 0.
\end{align}

Then, replacing \eqref{nablaTTN2} and \eqref{nablaTTN3} on \eqref{nablaTTN1},
it follows that
\begin{equation}
\label{nablaTTN}\langle\nabla_{T} T, N \rangle= f\theta^{\prime}\left(
ff_{r}\theta^{\prime2 }- r^{\prime\prime}\right)  + \frac{r^{\prime}}{f}
\left(  f^{2}\theta^{\prime}\right)  ^{\prime}.
\end{equation}

We now compute $\langle\nabla_{V_{i}}V_{i}, N\rangle$:
\begin{align}
\langle\nabla_{V_{i}}V_{i}, N\rangle &  = -f\theta^{\prime}\langle
\nabla_{V_{i}}V_{i}, U\rangle+ \frac{r^{\prime}}{f} \langle\nabla_{V_{i}}
V_{i}, V\rangle\nonumber\\
&  = f\theta^{\prime}\langle V_{i} , \nabla_{V_{i}} U \rangle- \frac
{r^{\prime}}{f} \langle V_{i} , \nabla_{V_{i}} V\rangle\nonumber\\
&  = f\theta^{\prime}\langle V_{i}, \nabla_{U} V_{i} \rangle- \frac{r^{\prime
}}{f} \langle V_{i}, \nabla_{V} V_{i} \rangle\nonumber\\
&  = f\theta^{\prime}\frac{1}{2} U\left(  \| V_{i}\|^{2}\right)  -
\frac{r^{\prime}}{f}\frac{1}{2} V\left(  \| V_{i}\|^{2}\right) \nonumber\\
&  = f^{2}f_{r} \theta^{\prime2}\theta_{i+1}\dots\sin^{2}\theta_{n} -
r^{\prime2}\theta_{i+1}\sin^{2}\theta_{n-1}\sin\theta_{n}\cos\theta
_{n}\nonumber\\
&  = f^{2} \sin^{2}\theta_{i+1}\dots\sin^{2}\theta_{n} \left(  f_{r}
\theta^{\prime}- \frac{r^{\prime}}{f}\cot\theta_{n}\right)  .
\end{align}

On the other hand, since $r$ and $\theta$ satisfy \eqref{rtheta}, we have
$\cos\theta=\tanh kR\coth kr$ and then
\begin{align*}
ff_{r}\theta^{\prime2}-r^{\prime\prime}  &  =ff_{r}\theta^{\prime2}
-\theta^{\prime}\cosh kR\cos\theta\\
&  =\theta^{\prime}\left(  ff_{r}\frac{k\sinh kR}{\sinh^{2}kr}-\cosh kR\coth
kr\tanh kR\right) \\
&  =\theta^{\prime}\left(  ff_{r}\frac{k\sinh kR}{\sinh^{2}kr}-\sinh
kR\frac{\cosh kr}{\sinh kr}\right) \\
&  =\frac{\theta^{\prime}}{\sinh^{2}kr}k\sinh kR\left(  ff_{r}-\cosh
kr\frac{\sinh kr}{k}\right)  \geq0
\end{align*}
where the last inequality is true by Lemma \ref{comparisionrotsym} (see
Section \ref{Appendices} below). Furthermore, $\theta^{\prime\prime
}=-2r^{\prime}\theta^{\prime}k\coth kr$ and then
\begin{align*}
r^{\prime}\left(  f^{2}\theta^{\prime}\right)  ^{\prime} =  &  r^{\prime
}\left(  2ff_{r}r^{\prime}\theta^{\prime} + f^{2}\theta^{\prime\prime}\right)
\\
=  &  2r^{\prime2}\theta^{\prime}\left(  ff_{r} - f^{2} k \coth kr\right)  =
2r^{\prime2}\theta^{\prime} f^{2} \left(  \frac{f_{r}}{f} - k \coth kr\right)
\ge0
\end{align*}
where the last inequality is again consequence of Lemma
\ref{comparisionrotsym}. Hence $\langle\nabla_{T}T,N\rangle\geq0$.

Moreover
\begin{align*}
f_{r}\theta^{\prime}-\frac{r^{\prime}}{f}\cot\theta &  =k\sinh kR\frac{f_{r}
}{\sinh^{2}kr}-\cosh kR\frac{\cos\theta}{f}\\
&  =\sinh kR\left(  k\frac{f_{r}}{\sinh^{2}r}-\frac{\coth kr}{f}\right) \\
&  =\frac{k\sinh kR}{f\sinh^{2}kr}\left(  ff_{r}-\cosh kr\frac{\sinh kr}
{k}\right)  \geq0,
\end{align*}
where the last inequality again comes from $f(r)\geq(\sinh kr)/k$ and
$f^{\prime}(r)\geq\cosh kr$. Therefore $\langle\nabla_{V_{i}}V_{i}
,N\rangle\geq0$ for all $i\in\{1,\dots,n-1\}$ and then $\Omega^{\prime}$ is convex.
\end{proof}

\begin{proof}
[Proof of Lemma \ref{SRhyperbolic}]In the hyperbolic space we have
$f(r)=(\sinh kr)/k$. If we replace $f(r)$ by $(\sinh kr)/k$ on the computation
done on the proof of the preceding proposition, it is easy to see that
$\left\langle \nabla_{T}T,N\right\rangle =0=\left\langle \nabla_{V_{i}}
V_{i},N\right\rangle $. To conclude, notice that $\cos\theta(t)=\tanh kR\coth
kr(t)$ implies that
\[
\lim_{t\rightarrow+\infty}\cos\theta(t)=\tanh kR,\lim_{t\rightarrow-\infty
}\cos\theta(t)=-\tanh kR,
\]
which implies that
\[
\partial_{\infty}S_{R}=\left\{  \tilde{\gamma}(\infty)\,|\,\tilde{\gamma
}(0)=o,\sphericalangle_{o}(\tilde{\gamma}^{\prime}(0) ,\gamma^{\prime
}(0))=\arccos\left(  \tanh kR\right)  \right\}  .
\]

\end{proof}

\begin{theorem}
\label{rotsim}If $M$ is a rotationally symmetric Hadamard manifold with
$K_{M}\leq-k^{2}$ then $M$ satisfies the SC condition.
\end{theorem}

\begin{proof}
Let $x\in\partial_{\infty} M$ and $W\subset\partial_{\infty} M$ relatively
open subset such that $x\in\mathrm{Int}\, W$. Let $\gamma:[0,+\infty)\to M$
the geodesic unit speed ray such that $\gamma(0)=o$ and $\gamma(\infty)=x$ and
let $\alpha\in(0,\pi/2)$ be such that
\[
\{\widetilde{\gamma}(\infty)\,|\, \widetilde{\gamma}(0)=o, \sphericalangle
(\widetilde{\gamma}^{\prime}(0),\gamma^{\prime}(0)) \le\alpha\}\subseteq W.
\]
Chose $R>>0$ in such a way that $\tanh kR =\cos\alpha$ and construct $S_{R}$ e
$\Omega^{\prime}$ as above. Setting $\Omega:= M\setminus\overline
{\Omega^{\prime}}$, Proposition \ref{SRgeneral} give us that $M\setminus
\Omega$ is convex and furthermore
\[
\partial_{\infty}\Omega=\left\{ \widetilde{\gamma}(\infty)\,|\, \widetilde
{\gamma}(0)=o,\sphericalangle(\widetilde{\gamma}^{\prime}(0),\gamma^{\prime
}(0))\le\alpha\right\} ,
\]
since the angles at the origin $o$ do not depend on the function $f$ that
defines the metric, and hence $\partial_{\infty} S_{R}$ is given by
\eqref{abSR}. It follows that $\Omega$ satisfies all desired conditions, which
concludes the proof.
\end{proof}

We now consider the case that $M$ is a general Hadamard manifold with
$K_{M}\leq-k^{2},$ $k>0,$ with a controll on the decay of the sectional
curvature. We prove:

\begin{theorem}
\label{bor} Let $k>0$ be given. Suppose that $K_{M}\leq-k^{2}$ and that there
exists $o\in M$ such that
\begin{equation}
K_{R}:=\min\{K_{M}(\Pi);\text{ }\Pi\text{ is a }2-\text{plane in }
T_{p}M,\text{ }p\in B_{R+1}(o)\} \label{aR}
\end{equation}
satisfies
\begin{equation}
\label{aRcondition}K_{R}\geq-\frac{e^{2kR}}{R^{2+2\epsilon}},\,\,R\geq
R^{\ast},
\end{equation}
for some constants $\epsilon,$ $R^{\ast}>0$.\ Then, given $\gamma(\infty
)\in\partial_{\infty} M$ with $\gamma(0)=o$ and $0<\alpha<\pi/2$ there exists
a convex set $C\subset M$ such that

\begin{enumerate}
\item[(i)] $\partial_{\infty}C\supseteq\{\tilde{\gamma}
(\infty)|\sphericalangle_{o}(\tilde{\gamma}(\infty),\gamma(\infty))\geq
2\alpha\}$

\item[(ii)] $C\cap T(\gamma^{\prime}(0),\alpha,r_{0}+1)=\emptyset$
for some $r_{0}\geq R^{\ast}$ large enough.
\end{enumerate}

In particular, $M$ satisfies the\textrm{ SC} condition.
\end{theorem}

We observe that it is easy to check that if \eqref{aRcondition} holds for some
point $o$ then it holds for all points, changing the constant $R^{\ast}$ if necessary.

For the proof of Theorem \ref{bor} we follow the construction of Anderson
\cite{A} and Borb\'ely \cite{B1}. Both results are concerned about the
existence of manifolds that satisfy the convex conic neighborhood condition of
Choi. Since we are interested in a similar but distinct result than Anderson
and Borb\'ely, we need to rewrite the proof with some adaptations. The
principal idea is that since the sectional curvature is bounded from above by
$-k^{2}$, $k>0$, all geodesic spheres have principal curvatures greater then
$k$ (if we orient them inwards). It allows us to take out a small piece of the
geodesic spheres and the remaining set is still convex. We observe that
Borbely requires a somewhat stronger condition, namely: $\inf_{B_{R}}K_{M}
\geq-e^{\lambda R},$ $R\geq R^{\ast},$ for some $R^{\ast}$ and $0<\lambda
<1/3.$

We observe that I. Holopainen and A. V\"{a}h\"{a}kangas in \cite{HV} prove
that the $p-$Laplacian is regular at infinity under the same hypothesis on the
curvature of $M$, namely, $K_{M}\leq-k^{2}$ and condition (\ref{aRcondition})
(see Theorem 3.21 and Corollary 3.23 of \cite{HV}). It is interesting to note
that the authors in \cite{HV}, using a quite different technique, arrived at
Corollary 3.23 requiring, as we did, the same decay condition for the
sectional curvature of $M.$

\bigskip
We start with an arbitrary choice of a smooth function
$\phi:[0,+\infty)\rightarrow\mathbb{R}$ such that $0\leq\phi\leq1$,
$\phi([0,1/2])=0$, $\phi\equiv1$ on $[1,+\infty)$ and $\phi^{\prime}\ge0$. Let
$L>0$ be such that $\phi^{\prime},\phi^{\prime\prime}\leq L$. For $p\in
S_{R}(o)$, let $f_{p}:M\rightarrow\mathbb{R}$ be given by
\[
f_{p}(x)=\phi(\rho_{p}(x)),
\]
where $\rho_{q}$ denotes the distance function to $q\in M$.

From now on, we use the notation $a_{R}:=\sqrt{-K_{R}}$.

\begin{lemma}
\label{betaL} There exists $\beta>0$ which depends only on $L$ and $k$ such
that if
\begin{equation}
\varepsilon_{R}:=\beta R^{1+\epsilon}e^{-kR},\label{epsilonR}
\end{equation}
then the sublevel set
\[
\{x\in M;\text{ }(\rho_{o}-\varepsilon_{R}f_{p})(x)\leq R\}
\]
is a strictly convex set for all $R\geq\max\{1,R^{\ast}\}$ and $p\in S_{R}(o)$.
\end{lemma}

\begin{proof}
It suffices to prove that the Hessian of $\rho_{o}-\varepsilon_{R}f_{p}$ is
positive-definite for all $X\perp\nabla(\rho_{o}-\varepsilon_{R}f_{p})$. Let
$X\perp\nabla(\rho_{o}-\varepsilon_{R}f_{p})$. Then
\begin{align*}
D^{2}(\rho_{o}-\varepsilon_{R}f_{p})(X,X)    =\langle\nabla_{X}\left(
\nabla\rho_{o}-\varepsilon_{R}\phi^{\prime}\nabla\rho_{p}\right)  ,X\rangle\\
  =D^{2}\rho_{o}(X,X)-\varepsilon_{R}\phi^{\prime}D^{2}\rho_{p}
(X,X)-\varepsilon_{R}\phi^{\prime\prime}\langle\nabla\rho_{p},X\rangle^{2}.
\end{align*}
Since $\rho_{o}$ is a distance function, we may have 
$$D^{2}\rho_{o}(X,X)=D^{2}\rho_{o}(X^{\perp},X^{\perp}),$$ where $X^{\perp}=X-\langle
X,\nabla\rho_{o}\rangle\nabla\rho_{o}$; furthermore, $\langle X,\nabla\rho
_{o}\rangle=\varepsilon_{R}\phi^{\prime}\langle X,\nabla\rho_{p}\rangle$.
Finally, since $K_{M}\leq-k^{2}$, the Hessian Comparison Theorem implies the
estimative $D^{2}\rho_{o}(X^{\perp},X^{\perp})\geq k \Vert X^{\perp}\Vert^{2}
$, and therefore
\begin{align}
D^{2}\rho_{o}(X,X) & \geq k\left( \Vert X\Vert^{2}- \langle X, \nabla\rho
_{o}\rangle^{2}\right) \nonumber\\
& = k\left( \Vert X\Vert^{2} - \varepsilon_{R} \phi^{\prime}\langle X,
\nabla\rho_{p} \rangle\langle X, \nabla\rho_{o} \rangle\right)  \ge
k(1-\varepsilon_{R} L) \Vert X\Vert^{2}.\label{D2rhoo}
\end{align}
On the other hand, since $\phi^{\prime}\equiv0$ on $[1,+\infty)$ and
$K_{M}\geq-a_{R}^{2}$ in $B_{R+1}(o)$, it is possible to apply the Hessian's
Comparison Theorem in the other direction to obtain
\begin{align*}
\phi^{\prime}D^{2}\rho_{p}(X,X)  &  \leq\phi^{\prime}a_{R}\coth(a_{R}\rho
_{p})\Vert X-\langle X,\nabla\rho_{p}\rangle\nabla\rho_{p}\Vert^{2}\\
&  =\phi^{\prime}a_{R}\coth(a_{R}\rho_{p})\left(  \Vert X\Vert^{2}-\langle
X,\nabla_{p}\rangle^{2}\right) \\
&  \leq\phi^{\prime}a_{R}\coth(a_{R}\rho_{p})\Vert X\Vert^{2}.
\end{align*}
Since we have $\phi^{\prime}\equiv0$ on $[0,1/2]$, with $a_{R}\geq1$, we also
obtain
\begin{equation}
\phi^{\prime}D^{2}\rho_{p}(X,X)\leq\phi^{\prime}a_{R}\coth(a_{R}/2)\Vert
X\Vert^{2}\leq L a_{R}\coth(k/2)\Vert X\Vert^{2}. \label{D2rhop}
\end{equation}
Therefore, if $\Vert X\Vert=1$, using that $a_{R} \le e^{kR}R^{-(1+\epsilon)}
$, it holds that
\begin{align}
D^{2}(\rho_{o}-\varepsilon_{R}f_{p})(X,X)    \geq k(1 -\varepsilon
_{R}L)-\varepsilon_{R}La_{R}\coth(k/2)-\varepsilon_{R}L\nonumber\\
  = k(1-\beta e^{-kR}R^{1+\epsilon} L) - \beta e^{-kR}R^{1+\epsilon} L (a_{R}
\coth(k/2) - 1)\nonumber\\
  \geq k-\beta\left( m(k+1) + \coth(k/2)\right) L,\nonumber
\end{align}
where $m:=\max\{e^{-kR}R^{1+\epsilon}\,|\, R\ge1\}$, and then, if we choose
\begin{equation}
\beta:=\frac{k}{2L(m(k+1)+\coth(k/2))},\label{beta}
\end{equation}
we obtain $D^{2}(\rho_{o}-\varepsilon_{R}f_{p})(X,X)\geq k/2$, which concludes
the proof.
\end{proof}

\bigskip
\begin{proof}
[Proof of Theorem \ref{bor}]We start with $r_{0}\geq R^{\ast}$ to be latter
determined, and let $v:=\gamma^{\prime}(0)\in T_{o}M$. Define the following
sets:
\[
C_{0}:=B_{r_{0}},\,\,\,\,\,\,T_{0}:=\{p\in S_{r_{0}};\sphericalangle
(\gamma_{op},v))<\alpha\},\,\,\,\,\,\,D_{0}=\emptyset
\]
Due to Lemma \ref{betaL}, there exists an uniform $\varepsilon_{0}
:=\varepsilon_{r_{0}}$ such that for each $p\in S_{r_{0}}(o)$, the sublevel
set
\[
C_{1,p}:=\{x\in M;\text{ }(\rho_{o}-\varepsilon_{0}f_{p})(x)\leq r_{0}\}
\]
is a convex set. We therefore define the second collection of sets as
follows:
\begin{align*}
&  \widetilde{C_{1}}:=\bigcap_{p\in T_{0}}C_{1,p},\,\,\,\,\,\,C_{1}
:=\widetilde{C_{1}}\setminus D_{0},\,\,\,\,\,\,r_{1}:=r_{0}+\varepsilon_{0}\\
&  D_{1}:=B_{r_{1}}(0)\setminus C_{1},\,\,\,\,\,\,T_{1}:=\overline{S_{r_{1}
}(o)\setminus\partial C_{1}}.
\end{align*}
Notice that $C_{1}$ is convex, since it is the intersection of convex sets.

We also define
\[
\theta_{0}:=\sup\{\sphericalangle(\gamma_{op}^{\prime}(0),\gamma_{oq}^{\prime
}(0));\text{ }p\in T_{0},\text{ }q\in S_{1}(p)\}.
\]
Therefore, $\sphericalangle(v,\gamma_{ox}^{\prime}(0))\leq\alpha+\theta_{0}$,
for all $x\in D_{1}$. In fact, if $x\in D_{1}$, there exists $p\in T_{0}$ such
that $x\in B(r_{1})\setminus C_{1,p}$, which implies that $r_{0}
+\varepsilon\phi(\rho_{p}(x))<\rho(x)<r_{1}$, and then $x\in B_{1}(p)$,
otherwise we would conclude that $r_{1}<r_{1}$. Hence
\[
\sphericalangle(v,\gamma_{ox}^{\prime}(0))\leq\sphericalangle(v,\gamma
_{op}^{\prime}(0))+\sphericalangle(\gamma_{op}^{\prime}(0),\gamma_{ox}
^{\prime}(0))\leq\alpha+\theta_{0}.
\]

We now proceed inductively on the following way: Assume that $C_{k}
,r_{k},D_{k},T_{k}$ are defined for all $k\leq n-1$, $n\geq1$, and that all
$C_{k}$'s are convex sets. By Lemma \ref{betaL}, there exists uniform
$\varepsilon_{n}$ such that if $p\in S_{r_{n-1}}(o)$, the set $C_{n,p}
:=\{\rho_{o}-\varepsilon_{n-1}f_{p}\leq r_{n-1}\}$ is convex. Define
\begin{align*}
&  \widetilde{C_{n}}:=\bigcap_{p\in T_{n-1}}C_{n,p},\,\,\,\,\,\,C_{n}
:=\widetilde{C_{n}}\setminus D_{n-1},\,\,\,\,\,\,r_{n}:=r_{n-1}+\varepsilon
_{n-1}\\
&  D_{n}:=B_{r_{n}}(0)\setminus C_{n},\,\,\,\,\,\,T_{n}:=\overline{S_{r_{n}
}(o)\setminus\partial C_{n}},\\
&  \theta_{n}:=\sup\{\sphericalangle(\gamma_{op}^{\prime}(0),\gamma
_{oq}^{\prime}(0));p\in T_{n},q\in S_{1}(p)\}.
\end{align*}

\textbf{Claim 1.} $C_{n}$ is a convex set.

For, let $x,y\in C_{n}$ and assume, by contradiction, that $\gamma
_{xy}\nsubseteq C_{n}$. Since $C_{n}\subset\widetilde{C_{n}}$, which is
convex, it must occur that $\gamma_{xy}\cap D_{n-1}\neq\emptyset$. In
particular, $\gamma_{xy}$ must intersect $\partial D_{n-1}$ at least twice. On
the other hand, since $C_{n-1}$ is a convex set, $\gamma_{xy}$ cannot cross
$\partial C_{n-1}$ twice, otherwise it would be contained in $C_{n-1}\subset
C_{n}$. Hence $\gamma_{xy}$ must intersect $\partial D_{n-1}\setminus\partial
C_{n-1}=\mathrm{{int}\,}T_{n-1}$. But in that case, $\gamma_{xy}$ would
contain points of $B_{r_{n}}(o)\setminus\widetilde{C_{n}}$, which contradicts
the convexity of $\widetilde{C_{n}}$.

\textbf{{Claim 2.}} If $x\in D_{n}$, then $\sphericalangle(v, \gamma
_{ox}^{\prime}(0)) \le\alpha+ \sum_{i=0}^{n-1} \theta_{i}$.

In fact, the definition of $D_{n}$ implies that $D_{n}=D_{n-1} \cup\left(
B_{r_{n}}(o) \setminus\widetilde{C_{n}}\right)  $. If $x\in D_{n-1}$, it is
finished, since by induction it is possible to conclude that $\sphericalangle
(v, \gamma_{oy}^{\prime}(0)) \le\alpha+ \sum_{i=0}^{n-2} \theta_{i}$ for all
$y\in D_{n-1}$. Otherwise, there exists $p\in T_{n-1}$ such that $x\in
B_{r_{n}}(o) \setminus C_{n,p}$, which implies that $x\in B_{r_{n}}(o) \cap
B_{1}(p)$ for some $p\in T_{n-1}$. Then it holds that
\begin{align*}
\sphericalangle(v, \gamma_{ox}^{\prime}(0))  &  \le\sphericalangle(v,
\gamma_{op}^{\prime}(0)) + \sphericalangle(\gamma_{op}^{\prime}(0),\gamma
_{ox}^{\prime}(0))\\
&  \le\sphericalangle(v, \gamma_{op}^{\prime}(0)) + \theta_{n-1}.
\end{align*}
On the other hand, it is easy to see that $T_{n-1} \subset\overline{D_{n-1}}$,
hence $$\sphericalangle(v, \gamma_{op}^{\prime}(0)) \le\alpha+ \sum_{i=0}^{n-2}
\theta_{i},$$ which concludes the proof of Claim 2.

\textbf{Claim 3.} Let
\[
C:=\bigcup_{n\ge0} C_{n},\,\,\,\,\,\, D:=\bigcup_{n\ge0} D_{n}.
\]
Then $C$ is a convex set and $M=C\cup D$.

Since by construction we have $C_{0}\subseteq C_{1}\subseteq C_{2}
\subseteq\cdots$ and all $C_{n}$s are convex, it is follows immediately that
$C$ is convex. In order to prove that $M=C\cup D$, it is sufficient to show
that $r_{n}\rightarrow+\infty$ as $n\rightarrow+\infty$ because $B_{r_{n}
}(o)=C_{n}\cup D_{n}$. If we assume, by contradiction, that there exists $A>0$
such $r_{n}\leq A$ for all $n\geq0$, we conclude that
\[
\varepsilon_{n}=\beta r_{n}^{1+\epsilon}e^{-kr_{n}}\geq\beta r_{0}
^{1+\epsilon} e^{-kA},\,\,\forall n\geq0,
\]
and using that
\[
r_{n+1}=r_{n}+\varepsilon_{n}=r_{0}+\sum_{i=0}^{n}\varepsilon_{n},
\]
we conclude that $r_{n}\rightarrow+\infty$, which leads to an absurd.

Because of Claim 3, to obtain \textrm{{(i)}} it suffices to prove that 
$$\partial_{\infty}D\subset\{\tilde{\gamma}(\infty)|\sphericalangle_{o}(v,\tilde{\gamma}(\infty))\leq2\alpha\}.$$ 
 For, notice that, at least
formally, if $x\in D$,
\[
\sphericalangle_{o}(v,\gamma_{ox}^{\prime}(0))\leq\alpha+\sum_{n=0}^{+\infty
}\theta_{n}.
\]
The only thing that remains to be proved is that the series converges to a
constant $\leq\alpha$ if we choose $r_{0}$ correctly. The main idea is to
estimate how many $r_{i}$'s are on each interval in order to control, in some
sense, the number of terms of the sum. For, let $t_{n}$ be the number of
$r_{i}$'s on the interval $I_{n}:=[r_{0}+n,r_{0}+n+1]$, $n\geq0$. Let $j_{n}$
be the index of the greatest $r_{i}$ on $I_{n}$. Since $\varepsilon_{i}$ is
decreasing on $i$, we have:
\begin{align*}
t_{n}\varepsilon_{j_{n}}  &  \leq\varepsilon_{j_{n}-1}+\varepsilon_{j_{n}
-2}+\cdots+\varepsilon_{j_{n}-t_{n}}\\
&  =(r_{j_{n}}-r_{j_{n}-1})+(r_{j_{n}-1}-r_{j_{n}-2})+\cdots+(r_{j_{n}
+1-t_{n}}-r_{j_{n}-t_{n}})\\
&  =r_{j_{n}}-r_{j_{n}-t_{n}}\leq(r_{0}+n+1)-(r_{0}+n)=1.
\end{align*}
It implies that
\[
t_{n}\leq\varepsilon_{j_{n}}^{-1}\leq\frac{e^{k(r_{0}+n+1)}}{\beta
(r_{0}+n)^{1+\epsilon}},
\]
where the last inequality is true since $r_{0}+n\leq r_{j_{n}}\leq r_{0}+n+1$.

Combining the above estimative with Lemma \ref{halfviewingangle} (see
Section\ref{Appendices}), we obtain that
\begin{align*}
\sum_{i=0}^{\infty}\theta_{i}  &  =\sum_{n=0}^{\infty}\sum_{r_{i}\in I_{n}
}\theta_{i} \leq\sum_{n=0}^{\infty}t_{n}Ce^{-k(r_{0}+n)}\leq\frac{Ce}{\beta}
\sum_{n=0}^{\infty}\frac{1}{(r_{0}+n)^{1+\epsilon}}\\
\end{align*}
It is then clear that with a correct choice of $r_{0}$ it is possible to
obtain $\theta<2\alpha$.

To prove \textrm{{(ii)}}, let $\tilde{\gamma}$ be a geodesic ray from $o$
satisfying $\sphericalangle_{o}(\tilde{\gamma}^{\prime}(0),v)<\alpha$. Then
$\tilde{\gamma}$ must intersect $T_{0}$ at distance $r_{0}$, going out of $C$.
Since by construction it holds that $\partial C\supseteq T_{0}$ and $C$ is
convex, $\tilde{\gamma}$ can not intersect $C$ after a distance $r_{0}$.
\end{proof}

\subsection{\label{Appendices}Appendix}

\begin{lemma}
\label{comparisionrotsym} Let $M$ be a rotationally symmetric manfold, with
metric given by
\[
dr^{2}+f(r)^{2}(d\theta^{2}+\sin^{2}\theta d\varphi^{2}),
\]
where $f:[0,+\infty)\to\mathbb{R}$ is a smooth function that satisfies
$f(0)=0$, $f^{\prime}(0)=1$ and $r$ is the distance function to the center $o$
of $M$. Suppose that $K_{M} \le-k^{2}<0$. Then for all $r\ge0$,
\begin{equation}
\frac{f^{\prime}(r)}{f(r)} \ge k\coth(kr),\,\,\,\,\, f(r)\ge\frac{\sinh
(kr)}{k}\text{ and } f^{\prime}(r) \ge\cosh(kr).
\end{equation}

\end{lemma}

\begin{proof}
Let $p\in M$ such that $r(p)=r$. Then it is well-known that the sectional
curvature of a $2$-plane generated by $\gamma_{op}^{\prime}(r)$ and any other
vector is given by $-f^{\prime\prime}(r)/f(r)$. Let $g(r):=(1/k)\sinh(kr)$.
Since $K_{M} \le-k^{2}$, it occurs that
\begin{align*}
\frac{f^{\prime\prime}(r)}{f(r)} \ge k^{2} = \frac{g^{\prime\prime}(r)}{g(r)}
\Rightarrow f^{\prime\prime}(r)g(r) - g^{\prime\prime}(r)f(r) \ge0\\
\Rightarrow\left(  f^{\prime}(r)g(r) - g^{\prime}(r)f(r)\right)  ^{\prime}
\ge0 \Rightarrow f^{\prime}(r)g(r) - g^{\prime}(r)f(r) \ge0,
\end{align*}
therefore
\begin{equation}
\label{comp1}\frac{f^{\prime}(r)}{f(r)}\ge\frac{g^{\prime}(r)}{g(r)} = k \coth
kr
\end{equation}
and also
\[
\left(  \frac{f(r)}{g(r)}\right)  ^{\prime} \ge0.
\]
The last inequality implies that
\[
\frac{f(r)}{g(r)} \ge\lim_{r\rightarrow0^{+}} \frac{f(r)}{g(r)} = 1
\Rightarrow f(r) \ge g(r) = \frac{\sinh kr}{k}.
\]
Using this fact on \eqref{comp1}, it is possible to conclude that $f^{\prime
}(r) \ge g^{\prime}(r)=\cosh kr$, which concludes the proof.
\end{proof}

\begin{lemma}
\label{halfviewingangle} Let $M$ be a Hadamard manifold with $K_{M}\leq-k^{2}
$, $o\in M$ and $p\in M$ such that $\rho_{o}(p)=R$, with $R\geq\tilde{r}$,
where $\tilde{r}$ is an universal constant. Then
\[
\theta_{R}:=\max\{\sphericalangle(\gamma_{op}^{\prime}\left(  0\right)
,\gamma_{oq}^{\prime}\left(  0\right)  );\text{ }q\in S_{1}(p)\}\leq
2\,\frac{\sinh k}{\sinh kR}.
\]
In particular, there exists an universal constant $C>0$ such that $\theta
_{R}\leq Ce^{-kR}$.
\end{lemma}

\begin{proof}
By Toponogov's Theorem, it suffices to prove that the estimative for
$\theta_{R}$ holds on the particular case that $M=\mathbb{H}^{2}(-k^{2})$, the
hyperbolic plane of sectional curvature $-k^{2}$. Let $\mathbb{D}^{2}$ be the
unitary disc $\{(x,y)\in\mathbb{R}^{2};$ $x^{2}+y^{2}<1\}$ endowed with the
conformal metric
\[
\langle u,v\rangle_{(x,y)}=\frac{4\langle u,v\rangle_{e}}{k^{2}(1-x^{2}
-y^{2})^{2}},
\]
where $\langle\cdot,\cdot\rangle_{e}$ denotes the Euclidean metric.

Since the hyperbolic space is homogeneous, we may assume without loss of
generality that $p$ has Euclidean coordinates $(a,0)$, $a>0$. An immediate
consequence of the expression of the metric is that $a=\tanh(kR/2)$. We
observe that the hyperbolic and Euclidean angles at $o=(0,0)$ coincide,
because the hyperbolic metric is conformal to the Euclidean one. Therefore, to
compute $\theta_{R}$ it suffices to compute the Euclidean maximal angle of the
circle $\mathbb{B}$ of hyperbolic center $(a,0)$ and hyperbolic radius $1$. It
is easy to see that $\theta_{R}$ is then given by
\begin{equation}
\sin\theta_{R}=\frac{\text{Euclidean radius of }\mathbb{B}}{\text{Euclidean
norm of the Euclidean center of }\mathbb{B}}. \label{sinthetaR}
\end{equation}
It then remains to compute the Euclidean radius and center of $\mathbb{B}$.
For, notice that if we write $z=x+iy$, the M\"{o}ebius transformation
\[
g(z):=\frac{z+a}{az+1}
\]
is an isometry of $\mathbb{D}^{2}$ that maps the horizontal axis $\{y=0\}$ on
itself and maps the hyperbolic circle $\widetilde{\mathbb{B}}$ centered at $o$
of radius $1$ (which is the Euclidean circle centered at $(0,0)$ of radius
$b:=\tanh(k/2)$ on the hyperbolic circle centered at $(a,0)$ with hyperbolic
radius $1$). Observe that the Euclidean and hyperbolic radius and center do
not coincide, however by the expression of $g$ is easy to see that $g(-b)$ and
$g(b)$ form a diameter of $\mathbb{B}$, from what we obtain after
straightforward computations that
\begin{align*}
\sin\theta_{R}  &  =\frac{g(b)-g(-b)}{g(b)+g(b)}=\frac{b(1-a^{2})}{a(1-b^{2}
)}\\
&  =\frac{\tanh(k/2)(1-\tanh^{2}(kR/2))}{\tanh(kR/2)(1-\tanh^{2}(k/2))}
=\frac{\tanh(k/2)\cosh^{2}(k/2)}{\tanh(kR/2)\cosh^{2}(kR/2)}\\
&  =\frac{\sinh(k/2)\cosh(k/2)}{\sinh(kR/2)\cosh(kR/2)}=\frac{\sinh k}{\sinh
kR}
\end{align*}
and then
\[
\theta_{R}=\arcsin\left(  \frac{\sinh k}{\sinh kR}\right)  ,
\]
and if $\sinh^{2}k/\sinh^{2}kR\leq3/4$, we have $\sqrt{1-\sinh^{2}k/\sinh
^{2}kR}\geq1/2$, which implies that
\[
\arcsin\frac{\sinh k}{\sinh kR}=\int_{0}^{\frac{\sinh k}{\sinh kR}}\frac
{dt}{\sqrt{1-t^{2}}}\leq\int_{0}^{\frac{\sinh k}{\sinh kR}}2dt,
\]
which gives the result.

Furthermore, if $R\geq(\ln{2})/2k$, we obtain that $\sinh kR\geq e^{kR}/4$,
and hence if we write $C:=8\sinh(k)$, we obtain the desired estimative, with
$\tilde{r}$ satisfying both conditions $\sinh^{2}k\tilde{r}\geq4\sinh
^{2}(k)/3$ and $\tilde{r}\geq(\ln{2})/2k$.
\end{proof}

\centerline{
\begin{tabular}{ccc}
{\small Jaime B. Ripoll} & $\hspace{1.5cm}$& {\small Miriam Telichevesky} \\
{\small UFRGS}  &  & {\small UFRGS}\\
{\small Instituto de Matem\'atica} &
& {\small Instituto de Matem\'atica}\\
{\small Av. Bento Gon\c calves 9500} & & {\small Av. Bento Gon\c calves 9500}\\
{\small 91540-000 Porto Alegre-RS } & &{\small 91540-000 Porto Alegre-RS }\\
{\small  BRASIL} &  & {\small BRASIL} \\
{\small jaime.ripoll@ufrgs.br}& &{\small miriam.telichevesky@ufrgs.br} \\
\end{tabular}}

\end{document}